\documentclass{amsart}

\usepackage[utf8]{inputenc}  
\usepackage[T1]{fontenc} 
\usepackage{amsmath,amssymb,amsthm} 
\usepackage[pagebackref=false]{hyperref} 
\usepackage{booktabs}
\usepackage{microtype} 
\usepackage{csquotes}
\usepackage[inline]{enumitem} 
    \setlist[enumerate]{label=\textup{(\roman*)}}
\usepackage{tikz}
\usepackage{nicematrix}
\usepackage[all,cmtip]{xy} 
\usepackage[T1]{fontenc}
\usepackage[utf8]{inputenc}
\usepackage{indentfirst}
\usepackage{enumitem}
\usepackage{amsthm,amsmath,amssymb}
\usepackage{mathrsfs}
\usepackage{geometry}
\usepackage{xcolor}
\usepackage{hyperref}
\usepackage{url}
\usepackage{comment}
\hypersetup{hypertex=true,
colorlinks=true,
linkcolor=blue,
anchorcolor=black,
citecolor=red}

\setenumerate{itemsep=0pt,topsep=0pt,parsep=0pt,partopsep=0pt}

\title[$p$-Adic $\mathbb{Z}^d$-odometers of adding type]{On free minimal constant speedups violating continuous orbit equivalence in $p$-adic $\mathbb{Z}^d$-odometers of adding type}
\author{Chang-Hua JIAO}
\date{\today}
\begin{document}
\newtheorem{theorem}{Theorem}[section]
\newtheorem{corollary}[theorem]{Corollary}
\newtheorem{lemma}[theorem]{Lemma}
\newtheorem{proposition}[theorem]{Proposition}

\theoremstyle{definition}
\newtheorem{definition}[theorem]{Definition}
\newtheorem{example}[theorem]{Example}
\newtheorem{conjecture}[theorem]{Conjecture}
\newtheorem{question}[theorem]{Question}

\theoremstyle{remark}
\newtheorem{remark}{Remark}

\theoremstyle{plain}
\renewcommand{\labelenumi}{(\theenumi)}
\renewcommand{\labelenumiii}{(\theenumiii)}

\begin{abstract}
Let $\mathbb{Z}_p$ be the ring of $p$-adic integers with respect to a prime $p$ and let $d$ be a positive integer.
For each $\mathbf{z}=(z_1, z_2,..., z_d) \in \mathbb{Z}_p^d$, let $T_{\mathbf{z}}: \mathbb{Z}^d \times \mathbb{Z}_p \to \mathbb{Z}_p$ be an adding-type $\mathbb{Z}^d$-action on $\mathbb{Z}_p$ defined by $T^{\mathbf{n}}_{\mathbf{z}}(x):=x+ \sum_{i=1}^d n_i z_i$ for $\mathbf{n}=(n_1,n_2, \cdots, n_d) \in \mathbb{Z}^d$ and $x \in \mathbb{Z}_p$.
Under some mild assumptions on $\mathbf{z}$, the action $T_{\mathbf{z}}$ is a free $\mathbb{Z}^d$-odometer (by odometer, we mean a minimal and equicontinuous action on a Cantor space).
In this paper, we derive a necessary condition for continuous orbit equivalence between such $\mathbb{Z}^d$-odometers by constructing algebraic models for them.
We then study the free minimal constant speedups of these $\mathbb{Z}^d$-odometers.
It turns out that such a speedup of $T_{\mathbf{z}}$ is again an adding-type $p$-adic $\mathbb{Z}^d$-odometer $T_{\mathbf{w}}$ for some $\mathbf{w} \in \mathbb{Z}_p^d$. 
However, the necessary condition above may not hold for the speedup.
This provides the first known examples of free minimal bounded speedups (of free $\mathbb{Z}^d$-odometers) which are not continuously orbit equivalent to the original ones and hence disproves a conjecture by Johnson and McClendon.
Our result also indicates that continuous orbit equivalence is a rare phenomenon for free minimal constant speedups of $p$-adic $\mathbb{Z}^d$-odometers of adding type when $d \geqslant 2$.
\end{abstract}
\maketitle

\tableofcontents

\section{Introduction}

\textbf{History and motivation.}
\emph{Orbit equivalence} (see Definition \ref{Definition_conjugacy_isomorphism_orbit_equivalence} for the topological version) is a pivotal notion in dynamical systems. 
Its development in ergodic theory (i.e., measurable dynamical systems) has a long history and is now relatively well-established.
Later, the study of orbit equivalence was extended to topological dynamics, with minimal Cantor systems emerging as a particularly rich testing ground, e.g. \cite{Topological_orbit_equivalence_and_C^*_crossed_products,Orbit_equivalence_for_Cantor_minimal_Z^2_systems,Orbit_equivalence_for_Cantor_minimal_Z^d_systems}.
In the topological setting, \emph{continuous orbit equivalence} (see Definition \ref{Definition_continuous_orbit_equivalence}) arises naturally, adhering to the principle that everything should respect continuity.
There are many studies related to continuous orbit equivalence; for example, see \cite{Bounded_topological_orbit_equivalence_and_C^*_algebras,Continuous_orbit_equivalence_rigidity,Amenable_minimal_Cantor_systems_of_free_groups_arising_from_diagonal_actions,Z^d_odometers_and_cohomology,On_Z^d-odometers_associated_to_integer_matrices}.

A notion similar to orbit equivalence is called \emph{speedup} (see Definition \ref{Definition_speedup} for topological version).
Roughly speaking, a speedup of a dynamical system is a new dynamical system derived from the original one through a certain procedure.
Research on speedups also initially focused on measurable dynamical systems, dating back to Neveu's work \cite{1_Neveu_1969,2_Neveu_1969} in 1969.
In the past decade, \emph{topological speedups} (i.e., speedups of topological dynamics) have gradually attracted increasing attention since Ash's PhD thesis \cite{Topological_Speedups_PhD_thesis}.

The similarity between orbit equivalence and topological speedup in minimal Cantor systems was further confirmed by the recent work of Ash and Ormes in \cite{Topological_speedups_for_minimal_Cantor_systems}.
Within this similarity, continuous orbit equivalence naturally corresponds to \emph{bounded speedups} (which is another equivalent term for \emph{continuous speedups} mostly, see Lemma \ref{Lemma_equivalent_conditions_for_bounded_speedups}) for $\mathbb{Z}^d$-actions.
Several results on bounded speedups of $\mathbb{Z}$-actions have already been established, mostly aiming to find the properties the speedups will inherit, see \cite{Bounded_topological_speedups,Bratteli_diagrams_for_bounded_topological_speedups,alvin2026minimalboundedspeedupstoeplitz,bruin2026speedupslinearlyrecurrentsubshifts}.
For example, it was shown by Alvin, Ash, and Ormes that a minimal bounded speedup of a $\mathbb{Z}$-odometer is topologically conjugate to the original one (see Theorem 3.3 in \cite{Bounded_topological_speedups}).

When speedups of $\mathbb{Z}^d$-actions are considered, things get much more complicated.
As proved by Johnson and McClendon in \cite{Topological_speedups_of_Z^d_actions},
a free minimal bounded speedup to $\mathbb{Z}^{d'}$ of a free $\mathbb{Z}^{d}$-odometer is again a free 
$\mathbb{Z}^{d'}$-odometer (see Theorem \ref{Theorem_Corrected_version_of_a_result_by_JM}, corrected version Theorem 5.3 in \cite{Topological_speedups_of_Z^d_actions}\footnote{The freeness assumption is missing from the original statement of Theorem 5.3 in \cite{Topological_speedups_of_Z^d_actions}, which makes it incorrect.
We give a correction to it as Theorem \ref{Theorem_Corrected_version_of_a_result_by_JM} in Subsection 2.4.}), while such speedup may not be conjugate (even not isomorphic) to the original one (see Theorem 5.5 therein).
They also showed that, for two continuously orbit equivalent $\mathbb{Z}^d$-odometers, one is always a minimal bounded speedup for the other one (see Theorem 5.7 therein).
For the converse, they left the following conjecture (in particular, for 2-dimensional case).

\begin{conjecture}[Revised version of Conjecture 5.6 in \cite{Topological_speedups_of_Z^d_actions}\footnote{In its original statement of Conjecture 5.6 in \cite{Topological_speedups_of_Z^d_actions}, there is no freeness assumption either, which makes it false in an obvious way.
So we add the freeness assumption for the speedup and thus make it more sensible.
}] \label{Conjecture_continuousorbit_equivalence_for_free_speedups}
For an integer $d \geqslant 2$, let $(X,T)$ be a $\mathbb{Z}^d$-odometer and $S: \mathbb{Z}^d \curvearrowright X $ be a free and minimal bounded speedup of $(X,T)$.
Then $(X,S)$ is continuously orbit equivalent to $(X,T)$.
\end{conjecture}

In this paper, we will borrow some ideas in $p$-adic dynamical systems to disprove this conjecture.
We consider a special class of $\mathbb{Z}^d$-odometers on $\mathbb{Z}_p$ (the ring of $p$-adic integers), which are called \emph{$p$-adic $\mathbb{Z}^d$-odometers of adding type}.
These $\mathbb{Z}^d$-odometers are natural high-dimensional analogues of the $p$-adic adding machine, a classical and well-studied $\mathbb{Z}$-odometer on $\mathbb{Z}_p$.
One motivation to consider such $\mathbb{Z}^d$-odometers comes from the author's collaborative work \cite{Affine_semigroup_dynamical_systems_on_Z_p}, where we studied the minimality of semigroup actions on $\mathbb{Z}_p$ of affine type, extending some minimality results in \cite{p_adic_affine_dynamical_systems_and_applications,Strict_ergodicity_of_affine_p_adic_dynamical_systems_on_Z_P} for affine $\mathbb{Z}$-actions on $\mathbb{Z}_p$. 
There are many other works on $p$-adic dynamical systems, for example \cite{Minimality_of_p_adic_rational_maps_with_good_reduction,On_minimal_decomposition_of_p_adic_homographic_dynamical_systems,On_minimal_decomposition_of_p_adic_polynomial_dynamical_systems,Minimal_polynomial_dynamics_on_the_set_of_3_adic_integers,Shadowing_and_stability_in_p_adic_dynamics,Ergodicity_criteria_for_non-expanding_transformations_of_2_adic_spheres}.
We also refer the readers to the classical monograph \cite{Applied_algebraic_dynamics} by Anashin and Khrennikov.

\textbf{$p$-Adic $\mathbb{Z}^d$-odometers of adding type.}
Here we briefly introduce the object of our study.
Some terms will be formally defined in the next preliminary section.

Let $\mathbb{Z}_p$ be the ring of $p$-adic integers with respect to the prime $p$ and $d$ be a positive integer.
For each $\mathbf{z}=(z_1, z_2,..., z_d)^t \in \mathbb{Z}_p^d$, let $T_{\mathbf{z}}: \mathbb{Z}^d \times \mathbb{Z}_p \to \mathbb{Z}_p$ be an adding-type $\mathbb{Z}^d$-action on $\mathbb{Z}_p$ defined by $$T^{\mathbf{n}}_{\mathbf{z}}(x):=x+\mathbf{n \cdot z}=x+ \sum_{i=1}^d n_i z_i$$ for $\mathbf{n}=(n_1,n_2, \cdots, n_d)^t \in \mathbb{Z}^d$ and $x \in \mathbb{Z}_p$.
Here the superscript $t$ denotes the transpose of the matrix (or in particular, the vector).

For $\mathbf{z}=(z_1, z_2, \cdots z_d)^t \in \mathbb{Z}_p^d \setminus \{\mathbf{0}\}$, put 
$$v_p(\mathbf{z}):= \min_{1 \leqslant i \leqslant d} v_p(z_i)=\max\{n \geqslant 0: p^n | z_i, \ \forall 1 \leqslant i \leqslant d\}.$$
By convention, put $v_p(\mathbf{0})=+\infty$.
Note that the $\mathbb{Z}^d$-action $(\mathbb{Z}_p, T_{\mathbf{z}})$ is minimal if and only if $v_p(\mathbf{z})=0$; i.e., $z_i \in \mathbb{Z}_p^{\times}$ is a $p$-adic unit for some $1 \leqslant i \leqslant d$.
Moreover, $(\mathbb{Z}_p, T_{\mathbf{z}})$ is free if and only if the coordinates of $\mathbf{z}$ are $\mathbb{Z}$-independent (or equivalently, $\mathbb{Q}$-independent) meaning that for all $\lambda_1, \lambda_2,...,\lambda_d \in \mathbb{Z}$ (or in $\mathbb{Q}$), we have $$\sum_{i=1}^d \lambda_i z_i =0 \iff \lambda_1=\lambda_2=\cdots =\lambda_d=0.$$
For short, we say $\mathbf{z}$ is $\mathbb{Z}$-\emph{independent}.

\textbf{Results and strategy.}
In this paper, we consider a special class of bounded speedups, called constant speedups (see Definition \ref{Definition_of_constant_speedup}).
They are perhaps the most tractable but non-trivial ones, since it is easy to see that constant speedups to $\mathbb{Z}^{d_2}$ of a $\mathbb{Z}^{d_1}$-action are in bijective correspondence with the matrices in $\mathrm{Mat}_{d_2 \times d_1}(\mathbb{Z})$.
Our main result below indicates that even for constant speedups, Conjecture \ref{Conjecture_continuousorbit_equivalence_for_free_speedups} is false; i.e., a free minimal constant speedup to $\mathbb{Z}^d$ of a free $\mathbb{Z}^d$-odometer is not necessarily continuously orbit equivalent to the original one.
Our counterexample to Conjecture \ref{Conjecture_continuousorbit_equivalence_for_free_speedups} is established via $p$-adic $\mathbb{Z}^d$-odometers of adding type, introduced above.
Indeed, we can prove that for such topological dynamical systems, Conjecture \ref{Conjecture_continuousorbit_equivalence_for_free_speedups} is false in a typical sense.

Recall that a property is said to hold for a \emph{generic} point in a topological space $X$ if the points in $X$ for which the property holds form a residual set, meaning that it contains a countable intersection of open dense subsets of $X$.
Our main result is as follows and will be finally proved as Theorem \ref{Theorem_main_theorem_residuality_of_exception_set} in Subsection 4.2 (with different but equivalent formulations).

\newtheorem*{theorema}{Theorem A} 
\begin{theorema} \label{Theorem_main_theorem}
    For every integer $d \geqslant 2$ and prime $p$, there is $\mathbf{C} \in \mathrm{Mat}_{d \times d}(\mathbb{Z})$ such that for generic $\mathbf{z} \in \mathbb{Z}_p^d \setminus p\mathbb{Z}_p^d$, we have 
    \begin{enumerate}
        \item $(\mathbb{Z}_p, T_{\mathbf{z}})$ is a free $\mathbb{Z}^d$-odometer;
        \item the constant speedup corresponding to $\mathbf{C}$ of $(\mathbb{Z}_p, T_{\mathbf{z}})$ is a free $\mathbb{Z}^d$-odometer, which is not continuously orbit equivalent to $(\mathbb{Z}_p,T_{\mathbf{z}})$.
    \end{enumerate}
\end{theorema}

As a consequence, we disprove Conjecture \ref{Conjecture_continuousorbit_equivalence_for_free_speedups}.

Theorem A (or its equivalent version, Theorem \ref{Theorem_main_theorem_residuality_of_exception_set}) is based on an auxiliary result stated as follows, which provides a necessary condition for continuous orbit equivalence between (free) $p$-adic $\mathbb{Z}^d$-odometers of adding-type.

\begin{theorem} \label{Theorem_necessary_condition_for_continuous_orbit_equivalence}
    Let $\mathbf{z}, \mathbf{w} \in \mathbb{Z}_p^d$ be $\mathbb{Z}$-independent with $v_p(\mathbf{z})=v_p(\mathbf{w})=0$.
    If the free $\mathbb{Z}^d$-odometers $(\mathbb{Z}_p, T_{\mathbf{z}})$ and $(\mathbb{Z}_p, T_{\mathbf{w}})$ are continuously orbit equivalent,
    then there is $\lambda \in \mathbb{Q}_p^{\times}=\mathbb{Q}_p \setminus\{0\}$ and $\mathbf{P} \in \mathrm{GL}_{d}(\mathbb{Q})$ with $\det(\mathbf{P})=\pm 1$ such that $\mathbf{w}=\lambda \mathbf{P} \mathbf{z}$.
\end{theorem}

This result will be proved in Section 3.
Our proof is based on the framework established in \cite{Z^d_odometers_and_cohomology} by Giordano, Putnam, and Skau.
To achieve this, we construct the algebraic model for the adding-type $p$-adic $\mathbb{Z}^d$-odometer $(\mathbb{Z}_p, T_{\mathbf{z}})$ (see Proposition \ref{Proposition_construction_of_algebraic_model}) and then compute $\mathcal{H}(\mathbb{Z}_p, T_{\mathbf{z}})$ (see Proposition \ref{Proposition_determining_first_cohomology_group}) which is the key (additive) subgroup of $\mathbb{Q}^d$, perhaps first introduced in \cite{Z^d_odometers_and_cohomology} to classify $\mathbb{Z}^d$-odometers up to continuous orbit equivalence (see our preliminary Subsection 2.3).

To apply general classification results in \cite{Z^d_odometers_and_cohomology} and \cite{On_Z^d-odometers_associated_to_integer_matrices}, one difficulty is that 
$\mathcal{H}(\mathbb{Z}_p, T_{\mathbf{z}})$ is a little bit complicated.
Our novel step is treating $\mathcal{H}(\mathbb{Z}_p, T_{\mathbf{z}})$ locally.
Formally, we consider the closure of $\mathcal{H}(\mathbb{Z}_p, T_{\mathbf{z}})$ in $\mathbb{Q}_p^d$ (vector space over $p$-adic field $\mathbb{Q}_p$ rather than $\mathbb{Q}$), denoted by $\mathcal{H}_p(\mathbb{Z}_p, T_{\mathbf{z}})$.
It turns out that $\mathcal{H}_p(\mathbb{Z}_p, T_{\mathbf{z}})$ is much easier: it is just the $\mathbb{Z}_p^{d}$-enlargement of the one-dimensional subspace along $\mathbf{z}$; i.e., $\mathcal{H}_p(\mathbb{Z}_p, T_{\mathbf{z}}) = \mathbb{Z}_p^d + \mathbb{Q}_p \mathbf{z}$ (see Proposition \ref{Lemma_determining_p_adic_closure_of_first_cohomology_group}).
This then easily leads to Theorem \ref{Theorem_necessary_condition_for_continuous_orbit_equivalence} by applying results in \cite{Z^d_odometers_and_cohomology} and \cite{On_Z^d-odometers_associated_to_integer_matrices}.

Note that the inverse of Theorem \ref{Theorem_necessary_condition_for_continuous_orbit_equivalence} is not true (see Section 5 for more discussion) and we leave it as an open problem to find a sufficient and necessary condition for this.

\textbf{Organization.} The rest of the paper is structured as follows.
In Section 2, we collect all preliminaries to $p$-adic numbers, topological dynamical systems, $\mathbb{Z}^d$-odometers, and topological speedups (with a correction to a result in \cite{Topological_speedups_of_Z^d_actions}).
Section 3 is devoted to the construction of the algebraic models for $p$-adic $\mathbb{Z}^d$-odometers of adding type and the proof of Theorem \ref{Theorem_necessary_condition_for_continuous_orbit_equivalence}.
In Section 4, we consider the free minimal constant speedups for $\mathbb{Z}^d$-odometers and prove our main result (Theorem A).
Finally, in the last section, we present an example to show our necessary condition in Theorem \ref{Theorem_necessary_condition_for_continuous_orbit_equivalence} is not sufficient.

\section{Preliminaries}

\subsection{$p$-adic numbers}

Let $p$ be a prime number. For a nonzero integer $n \in \mathbb{Z}$, denote by $v_p(n)$ the largest nonnegative integer $k$ such that $p^k$ divides $n$.
The $p$-adic absolute value on $\mathbb{Z}$ is defined by $|n|_p=p^{-v_p(n)}$ for $n \neq 0$ and $|0|_p = 0$.
It can be extended to $\mathbb{Q}$ by
$$\left|\frac{a}{b}\right|_p=\frac{|a|_p}{|b|_p}, \quad  \forall a,b \in \mathbb{Z},\ b \neq 0.$$
It is easy to check the following ultrametric inequality
$$|x+y|_p \leqslant \max\{|x|_p,|y|_p\}, \quad \forall x,y \in \mathbb{Q}.$$
Hence $|\cdot|_p$ induces a metric on $\mathbb{Q}$, called the $p$-adic metric.

The field of $p$-adic numbers $\mathbb{Q}_p$ is defined to be the completion of $\mathbb{Q}$ with respect to the $p$-adic metric.
Every non-zero element of $\mathbb{Q}_p$ can be represented by formal series
$z=\sum_{n=N}^{\infty} a_n p^n,$
where $N \in \mathbb{Z}$ and $a_n \in \{0,1, \dots ,p-1 \}$ for each $n \geqslant N$ and $a_N \neq 0$.
Note that now we have $v_p(z)=N$ and $|z|_p=p^{-N}$.
The ring of $p$-adic integers is defined by
$$\mathbb{Z}_p = \{x\in\mathbb{Q}_p:|x|_p\leqslant 1\}.$$
Note that every element of $\mathbb{Z}_p$ admits a unique expansion $z=\sum_{i \geqslant 0}a_ip^i$ with $ a_i \in \{0,1, \dots ,p-1\}$ for each $i \geqslant 0$.
This implies that $\mathbb{Z}_p = \varprojlim \mathbb{Z}/ p^n \mathbb{Z}$.
Note that
$\mathbb{Q}_p = \bigcup_{m \geqslant 0} p^{-m} \mathbb{Z}_p.$
In particular, every nonzero element $x \in \mathbb{Q}_p$ can be written uniquely
as $ x=p^k u,$
where $k \in \mathbb{Z}$ and $u \in \mathbb{Z}_p^{\times}= \mathbb{Z}_p \setminus p \mathbb{Z}_p$ is a $p$-adic unit.

The topology of $\mathbb{Q}_p$ is generated by the balls
$$B(a,p^{-n}) = \{x \in \mathbb{Q}_p:|x-a|_p \leqslant p^{-n}\}, \quad a \in \mathbb{Q}_p, \ n \in \mathbb{Z}. $$
Unlike the Euclidean case, these balls are both open and closed (i.e., clopen), and any two balls are either disjoint or one contains the other.
For $d \geqslant 1$, the product metric on $\mathbb{Q}_p^d$ is defined by $$\mathrm{dist}(\mathbf{z}, \mathbf{w})= \sum_{i=1}^{d} \mathrm{dist}(z_i, w_i)= \sum_{i=1}^{d} |z_i-w_i|_p,$$
where $\mathbf{z}=(z_1, z_2, \cdots, z_d)^{t}, \mathbf{w}=(w_1, w_2, \cdots, w_d)^{t} \in \mathbb{Q}_p^d$.
Note that $ K \subset \mathbb{Q}_p^d$ is compact if and only if $K$ is closed and bounded with respect to $\mathrm{dist}$.
In particular, $\mathbb{Z}_p^d$ is a compact open (additive) subgroup of $\mathbb{Q}_p^d$.

\subsection{Topological dynamical systems}
Let $G$ be a group. A (topological) $G$-system $(X,T)$ consists of a compact topological space $X$ and a $G$-action $T$ on $X$; i.e., $T^g \in \mathrm{Homeo}(X)$ for all $g \in G$. We also write $T: G \curvearrowright X$ or just $G \curvearrowright X$ when $T$ is clear. For a $G$-system $(X,T)$, the orbit of $x \in X$ is $Gx=Orb_{G}(x)=Orb_{T}(x)=\{T^gx:g \in G\}$. $T$ is said to be transitive if there is a dense orbit and is said to be minimal if every orbit is dense in $X$. 

$(Y,T)$ is said to be a $G$-subsystem of $(X,T)$ if $Y\subset X$ is nonempty and $G$-invariant, i.e., $T^g(Y)=Y,\ \forall g \in G$. Note that every $G$-system admits a minimal $G$-subsystem.

\begin{definition} \label{Definition_conjugacy_isomorphism_orbit_equivalence}
    Let $T: G \curvearrowright X$ and $S: H \curvearrowright Y$. \begin{itemize}
    \item $T$ and $S$ are \emph{(topologically) conjugate} if $G=H$ and there is a homeomorphism $\Phi:X \to Y$ such that $\Phi \circ T^g=S^g \circ \Phi, \ \forall g \in G=H.$
    \item $T$ and $S$ are \emph{isomorphic} if there is a group isomorphism $\tau: G \to H$ and a homeomorphism $\Phi:X \to Y$ such that  $\Phi \circ T^g=S^{\tau(g)} \circ \Phi, \ \forall g \in G.$
    \item $T$ and $S$ are \emph{orbit equivalent} if there is a homeomorphism $\Phi:X \to Y$ such that $\Phi(Gx)=H\Phi(x)$ for all $x \in X$.
\end{itemize}
\end{definition}

To further introduce continuous orbit equivalence, we first define freeness for dynamical systems.

\begin{definition}
    $T: G \curvearrowright X$ is said to be \emph{free} if for each $x \in X$,  $T^g x=x$ implies $g=e_G$ (the identity element of $G$).
\end{definition}

 Note that, once two free systems $T,S$ are orbit equivalent via a homeomorphism $\Phi: X \to Y$, there are two uniquely defined maps $\alpha:G \times X \to H$ and $\beta: H \times Y \to G$ such that $\Phi(T^gx)=S^{\alpha(g,x)}\Phi(x)$ and $\Phi^{-1}(S^hy)=T^{\beta(h,y)}\Phi^{-1}(y)$ for $g\in G,h\in H, x \in X, y \in Y$. These two maps $\alpha, \beta$ are called \emph{orbit cocycles} associated to $\Phi$.

\begin{definition} \label{Definition_continuous_orbit_equivalence}
    Suppose that $G,H$ are topological groups. Let $T: G \curvearrowright X$ and $S: H \curvearrowright Y$ be two free systems. If there is a homeomorphism $\Phi:X \to Y$ such that $T$ and $S$ are orbit equivalent via $\Phi$ and the corresponding orbit cocycles are continuous, we say $T$ and $S$ are \emph{continuously orbit equivalent}.
\end{definition}

Note that all equivalence relations introduced above preserve minimality.

\subsection{$\mathbb{Z}^d$-odometers}

By Cantor system, we mean a dynamical system whose underlying space is a Cantor space (i.e., a totally disconnected, metrizable, compact space without isolated points).

\begin{definition} [See, for example, Definition 2.2 in \cite{Orbit_equivalence_rigidity_of_equicontinuous_systems}] \label{Definition_odometers}
    A $G$-\emph{odometer} is an equicontinuous minimal Cantor $G$-system. 
    Here \emph{equicontinuity} means that $\{T^g : g \in G\} \subset \mathrm{Homeo}(X)$ is an equicontinuous family; that is, for any $\varepsilon>0$, there is $\delta>0$ such that $$\mathrm{d}_X(T^g x,T^g y)< \varepsilon, \quad \forall g \in G$$
    once $\mathrm{d}_X(x,y)< \delta$, where $\mathrm{d}_X$ is the metric on $X$.
\end{definition}

For basic discussions about $G$-odometers, we refer the readers to \cite{G_odometers_and_their_almost_one_to_one_extensions}.

It is known that (with mild assumptions on $G$) every $G$-odometer is conjugate to an inverse limit of a sequence of $G$-actions on quotient spaces determined by a sequence of finite-index subgroups of $G$ (See, for example, \cite{Type_invariants_for_non-abelian_odometers})
The latter one is called the \emph{algebraic model}.
We formalize this result when $G=\mathbb{Z}^d$ as follows. 
 
\begin{theorem} [See, for examples, Theorem 1.4 in \cite{Type_invariants_for_non-abelian_odometers} and Theorem 2.7 in \cite{Orbit_equivalence_rigidity_of_equicontinuous_systems}] \label{Theorem_algebraic_model_for_odometers}
     A $\mathbb{Z}^d$-system $(X,T)$ is a $\mathbb{Z}^d$-odometer if there is a decreasing sequence of subgroups $\mathfrak{G}=(G_n)_{n\geqslant1}$ of $\mathbb{Z}^d$, each of which has finite index, such that \begin{itemize}
        \item  $\lim_{n \to \infty}[\mathbb{Z}^d:G_n] =\infty$;
        
        \item (X,T) is conjugate to $$(X_{\mathfrak{G}},\sigma_{\mathfrak{G}}):=\varprojlim (\mathbb{Z}^d/G_n, \sigma_{n}),$$ where $(\mathbb{Z}^d/G_n, \sigma_{n})$ is a $\mathbb{Z}^d$-system given by $\sigma_n^{\mathbf{v}}(\mathbf{w}+G_n)=(\mathbf{v}+\mathbf{w})+G_n$.
    \end{itemize}
\end{theorem}

Note that the odometer $(X_{\mathfrak{G}},\sigma_{\mathfrak{G}})$ is free if and only if $\cap_{n \geqslant 1}G_n = \{ \mathbf{0} \}$.

We then review some results in \cite{Z^d_odometers_and_cohomology} and \cite{On_Z^d-odometers_associated_to_integer_matrices}, related to continuous orbit equivalence between free $\mathbb{Z}^d$-odometers.
To begin with, a \emph{lattice} $\Lambda \subset \mathbb{R}^d$ is a discrete additive subgroup with $\mathbb{R}^d/\Lambda$ compact.
 The \emph{dual} of a lattice $\Lambda$ is $$\Lambda^*:=\{\mathbf{v}\in \mathbb{R}^d: \mathbf{v\cdot w} \in \mathbb{Z},\ \forall \mathbf{w} \in \Lambda\},$$ which is clearly a lattice. 
For a free $\mathbb{Z}^d$-odometer $(X_{\mathfrak{G}}, \sigma_{\mathfrak{G}})$ w.r.t. $\mathfrak{G}=(G_n)_{n \geqslant 1}$, we put 
$$\mathcal{H}(X_{\mathfrak{G}}, \sigma_{\mathfrak{G}}):= \bigcup_{n \geqslant 1} G_n^*.$$
For a general free $\mathbb{Z}^d$-odometer $(X,T)$, we put $$\mathcal{H}(X,T):= \mathcal{H}(X_{\mathfrak{G}}, \sigma_{\mathfrak{G}})$$
where $(X_{\mathfrak{G}}, \sigma_{\mathfrak{G}})$ is an algebraic model of $(X,T)$.
It is well-defined according to Theorem 4.1 in \cite{On_Z^d-odometers_associated_to_integer_matrices} (see also Theorem 1.5 in \cite{Z^d_odometers_and_cohomology} for low dimensional cases), which asserts that two free odometers $(X_{\mathfrak{G}}, \sigma_{\mathfrak{G}})$ and $(X_{\mathfrak{H}}, \sigma_{\mathfrak{H}})$ are topologically conjugate if and only if $\mathcal{H}(X_{\mathfrak{G}}, \sigma_{\mathfrak{G}})=\mathcal{H}(X_{\mathfrak{H}}, \sigma_{\mathfrak{H}})$.

Note that $\mathbb{Z}^d \subset \mathcal{H}(X, T) \subset \mathbb{Q}^d$ as additive groups and $\mathcal{H}(X,T)$ is dense in $\mathbb{Q}^d$ since $(X,T)$ is free.
This group plays an important role in classifications of $\mathbb{Z}^d$-odometers.
For example, we will apply the following result later.

\begin{lemma}[Proposition 3.4 in \cite{On_Z^d-odometers_associated_to_integer_matrices}, see also Theorem 1.5 in \cite{Z^d_odometers_and_cohomology} for low dimensional cases] \label{Lemma_continuous_orbit_equivalence_for_Z^d_odometers}
   Let $(X_i,T_i)$ be free $\mathbb{Z}^{d_i}$-odometers ($i=1,2$). 
   Then they are continuously orbit equivalent if and only if $d_1=d_2=:d$ and there is $\mathbf{P} \in \mathrm{GL}_{d}(\mathbb{Q})$ with $\det(\mathbf{P})=\pm1$ such that $\mathbf{P}\mathcal{H}(X_1,T_1)=\mathcal{H}(X_2,T_2)$.
\end{lemma}

\subsection{Topological speedups}
Let $(G, \ast)$ and $(H, \star)$ be two groups and $(X,T)$ be a $G$-system.
A map $\kappa:X \times H \to G$ is called a \emph{speedup cocycle} for $(X,T)$ if $\kappa(x, e_H)=e_G$ and $$ \kappa(x,h_2 \star h_1)=\kappa(T^{\kappa(x,h_1)} x,h_2) \ast \kappa(x,h_1),$$
for all $h_1, h_2 \in H$ and $x \in X$.

\begin{definition} \label{Definition_speedup}
    An $H$-system $(X,S)$ is called a \emph{(topological) speedup} (to $H$) for a $G$-system $(X,T)$ if there is some speedup cocycle $\kappa: X \times H \to G$ such that $S^h(x)=T^{\kappa(x,h)}(x), \ \forall x \in X,h \in H$.
\end{definition} 

\begin{remark}
    Some people also call a dynamical system conjugate to $(X,S)$ the speedup of $(X,T)$, while we will follow our definition strictly.
\end{remark}

Our main focus is on (free) minimal bounded speedups.

\begin{lemma} [Lemma 5.2 in \cite{Topological_speedups_of_Z^d_actions}] \label{Lemma_equivalent_conditions_for_bounded_speedups}
    Let $(X,T)$ be a free Cantor $\mathbb{Z}^{d_1}$-system and $S$ be $\mathbb{Z}^{d_2}$-speedup of $T$ w.r.t the speedup cocycle $\kappa: X \times \mathbb{Z}^{d_2} \to \mathbb{Z}^{d_1}$. 
    Then the following conditions are all equivalent to each other:
    \begin{enumerate}
        \item The map $\kappa(\cdot, \mathbf{v}): X \to \mathbb{Z}^{d_1}$ is bounded for all $\mathbf{v} \in \mathbb{Z}^{d_2}$;
        \item The map $\kappa(\cdot, \mathbf{e}_j): X \to \mathbb{Z}^{d_1}$ is bounded for all $1 \leqslant j \leqslant d_2$;
        \item The map $\kappa(\cdot, \mathbf{v}): X \to \mathbb{Z}^{d_1}$ is continuous for all $\mathbf{v} \in \mathbb{Z}^{d_2}$;
        \item The map $\kappa(\cdot, \mathbf{e}_j): X \to \mathbb{Z}^{d_1}$ is continuous for all $1 \leqslant j \leqslant d_2$;
        \item The map $\kappa : X \times \mathbb{Z}^{d_2} \to \mathbb{Z}^{d_1}$ is continuous.
    \end{enumerate}
where $\mathbb{Z}^{d_1}$ and $\mathbb{Z}^{d_2}$ are equipped
with discrete topology and $\mathbf{e}_j=(0, \cdots ,0,1,0, \cdots, 0)^t$ is the $j$-th standard basis vector of $\mathbb{Z}^{d_2}$. 
\end{lemma}

When the conditions in the above lemma are fulfilled, we say $\kappa: X \times \mathbb{Z}^{d_2} \to \mathbb{Z}^{d_1}$ is \emph{bounded} and $S$ is a \emph{bounded speedup} of $T$.
If $S$ is minimal (and free) as a dynamical system and is a bounded speedup of $T$, we say it is a \emph{(free) minimal bounded speedup} of $T$.

\begin{definition} \label{Definition_of_constant_speedup}
   Let $(X,T)$ be a free Cantor $\mathbb{Z}^{d_1}$-system.
   A speedup cocycle $\kappa: X \times \mathbb{Z}^{d_2} \to \mathbb{Z}^{d_1}$ is said to be \emph{constant} if $\kappa(\cdot, \mathbf{v}): X \to \mathbb{Z}^{d_1}$ is constant for each fixed $\mathbf{v} \in \mathbb{Z}^{d_2}$.
   We say $S$ is a \emph{constant speedup} of $T$ if the corresponding speedup cocycle is constant.
\end{definition}

By definition, we see that a constant speedup is always bounded.
Moreover, it is easy to see that a speedup cocycle $\kappa: X \times \mathbb{Z}^{d_2} \to \mathbb{Z}^{d_1}$ is constant if and only if there is 
$\mathbf{C} \in \mathrm{Mat}_{d_1 \times d_2}(\mathbb{Z})$ such that $$\kappa(x,\mathbf{v})=\mathbf{C v}, \quad \forall x \in X, \mathbf{v} \in \mathbb{Z}^{d_2}.$$
The matrix $\mathbf{C}$ is said to be the \emph{speed} of $\kappa$ or the speedup.
We sometimes denote by $\kappa_{\mathbf{C}}$ the constant speedup cocycle w.r.t. $\mathbf{C}$.

In the rest of this subsection, we correct a minor mistake in Theorem 5.3 in \cite{Topological_speedups_of_Z^d_actions}.

\begin{definition}
    Let $G$ and $H$ be two groups and $(X,T)$ be a $G$-system.
    A speedup cocycle $\kappa: X \times H \to G$ is said to be \emph{non-vanishing} if $\kappa(x,h) \neq e_G$ for all $h \in H \setminus \{ e_H\}$.
\end{definition}

\begin{proposition} \label{Proposition_free_and_nonvanishing}
    Let $G$ and $H$ be two groups and $(X,T)$ be a free $G$-system.
    Suppose $(X,S)$ is a speedup of $(X,T)$ w.r.t. speedup cocycle $\kappa: X \times H \to G$.
    Then $(X,S)$ is free if and only if $\kappa$ is non-vanishing.
\end{proposition}

\begin{proof}
    Suppose $\kappa$ is non-vanishing.
    Let $h \in H \setminus \{e_H\}$.
    Then $\kappa(x,h) \neq e_G$ for all $x \in X$ by non-vanishment.
    Since $(X,T)$ is free, we see $$S^{h}(x)=T^{\kappa(x,h)}(x) \neq T^{e_G}(x)=x, \quad \forall x \in X.$$
    It follows that $(X,S)$ is free.

    Now suppose $\kappa(x,h)=e_G$ for some $x \in X$ and $h \neq e_H$ in $H$.
    We see $$S^h(x)=T^{\kappa(x,h)(x)}=T^{e_G}(x)=x,$$
    which means $(X,S)$ is not free.
\end{proof}

Theorem 5.3 in \cite{Topological_speedups_of_Z^d_actions} asserts that a minimal bounded speedup to $\mathbb{Z}^d$ of a free $\mathbb{Z}^d$-odometer is again a free $\mathbb{Z}^d$-odometer.
We claim that the non-vanishment has to be assumed; otherwise we have the following counterexample.

\begin{example} \label{Example_counterexample_to_Theorem_5.3_of_JM} 
    Let $\alpha$ be an irrational element in $\mathbb{Z}_p$ and put $\mathbf{z}=(1,\alpha)^t \in \mathbb{Z}_p^2$.
    Then it is easy to see $(\mathbb{Z}_p, T_{\mathbf{z}})$ is a free $\mathbb{Z}^2$-odometer.
    Put $$\mathbf{C}= \begin{pmatrix}
        1 &0\\
        0 &0\\
    \end{pmatrix} $$
    and let $S$ be the constant speedup of $(\mathbb{Z}_p, T_{\mathbf{z}})$ w.r.t. $\mathbf{C}$.
    Then we see $$S^{(n_1,n_2)^t} (x)= T_{\mathbf{z}}^{\mathbf{C}(n_1,n_2)^t}(x)= T_{\mathbf{z}}^{(n_1,0)^t}=x+n_1.$$
    So $(\mathbb{Z}_p, S)$ is minimal and thus it is a $\mathbb{Z}^2$-odometer.
    It is not free since $S^{(0,1)^t}=S^{(0,0)^t}.$
\end{example}

The gap in the original proof of Theorem 5.3 appears in its Claim 2 which is invalid without non-vanishment assumptions.
Adding this assumption, one can easily address the gap by following the lines of its original proof.
Below we present a corrected version of the result with a different proof.

\begin{theorem} [Corrected version of Theorem 5.3 in \cite{Topological_speedups_of_Z^d_actions}] \label{Theorem_Corrected_version_of_a_result_by_JM}
    Let $d_1, d_2 $ be two positive integers and $(X, T)$ be a free $\mathbb{Z}^{d_1}$-odometer. 
    Suppose $S: \mathbb{Z}^{d_2} \curvearrowright X$ is a minimal speedup of $(X,T)$ with respect to a bounded speedup cocycle $\kappa : X \times \mathbb{Z}^{d_2} \to \mathbb{Z}^{d_1}$.
    Then $(X, S)$ is a $\mathbb{Z}^{d_2}$-odometer.
    Moreover, $(X, S)$ is free if and only if $\kappa$ is non-vanishing.
\end{theorem}

\begin{proof}
    To see $(X, S)$ is a $\mathbb{Z}^{d_2}$-odometer, we only need to show equicontinuity for $S$ (since minimality is assumed).
    Let $\mathrm{d}_X$ be the metric on $X$.
    Define $$\mathrm{d}_X^T (x,y):=\sup_{\mathbf{n} \in \mathbb{Z}^{d_1}} \mathrm{d}(T^{\mathbf{n}}x, T^{\mathbf{n}}y).$$
    Then, by equicontinuity of $T$, we know $\mathrm{d}_X^T$ is a compatible metric with $\mathrm{d}_X$, meaning that they generate the same topology.
    Moreover, $\mathrm{d}_X^T$ is $T$-invariant; i.e., $d_X^T(T^{\mathbf{n}}x, T^{\mathbf{n}}y)=\mathrm{d}_X^T(x,y)$ for all $x,y \in X$ and $\mathbf{n} \in \mathbb{Z}^{d_1}.$

    Since $\kappa$ is bounded, by Lemma \ref{Lemma_equivalent_conditions_for_bounded_speedups} we see each $\kappa(\cdot, \epsilon\mathbf{e}_j): X \to \mathbb{Z}^{d_1}$ is continuous (w.r.t. $\mathrm{d}_X$ and thus w.r.t. $\mathrm{d}_X^T$) for all $1 \leqslant j \leqslant d_2$ and $\epsilon \in \{\pm 1\}$. 
    Since $X$ is a Cantor set, there is $\delta_0>0$ such that $$\mathrm{d}_X^T(x,y) < \delta_0 \implies \kappa(x, \epsilon \mathbf{e}_j)=\kappa(y,\epsilon \mathbf{e}_j), \quad \forall 1 \leqslant j \leqslant d_2, \ \epsilon \in \{\pm 1\}.$$
    Consequently, for all $x, y \in X$ with $\mathrm{d}_X^T(x,y)< \delta_0$ we have 
    $$\mathrm{d}_X^T(S^{\epsilon \mathbf{e}_j}x, S^{\epsilon \mathbf{e}_j}y)=\mathrm{d}_X^T \left(T^{\kappa(x, \epsilon \mathbf{e}_j)}x, T^{\kappa(y, \epsilon \mathbf{e}_j)}y \right) = \mathrm{d}_X^T(x,y), \quad \forall 1 \leqslant j \leqslant d_2, \ \epsilon \in \{\pm 1\}.$$
    Therefore, for all $x, y \in X$ with $\mathrm{d}_X^T(x,y)< \delta_0$, we see 
    $$\mathrm{d}_X^T(S^{\mathbf{n}}x, S^{\mathbf{n}}y)= \mathrm{d}_X^T(x,y), \quad \forall \mathbf{n} \in \mathbb{Z}^{d_2}.$$
    It then clearly means that $S$ is equicontinuous w.r.t. $\mathrm{d}_X^T$ and hence w.r.t. $\mathrm{d}_X$.

    The freeness assertion is just Proposition \ref{Proposition_free_and_nonvanishing}.
\end{proof}

\begin{remark}
    Our proof strategy is a well-known fact that an equicontinuous action is basically an isometric action, and an observation that a bounded speedup somehow preserves isometry locally.
    One may find that our proof is simpler than its original one in \cite{Topological_speedups_of_Z^d_actions} because ours is on account of the abstract Definition \ref{Definition_odometers} possessing more flexibility and generality.
    The original proof is based on the algebraic models (and K-R partitions) and says more things.
    For example, following the original proof, one can derive an algorithm to compute the (algebraic models of) speedups.
    Nevertheless, these two proofs are essentially equivalent in view of Theorem \ref{Theorem_algebraic_model_for_odometers}.
\end{remark}

\section{Continuous orbit equivalence for $(\mathbb{Z}_p, T_{\mathbf{z}})$}

\subsection{Construction of algebraic models}

Let $p$ be a prime and $\mathbb{Z}_p$ be the ring of $p$-adic integers.
For $z \in \mathbb{Z}_p$ and $n \geqslant 1$, we let $z^{(n)}$ be the $n$-th truncation of $z$, that is $z^{(n)} \in \mathbb{Z}/p^n \mathbb{Z}=\{0, 1, \cdots, p^n-1\}$ and $z^{(n)} \equiv z \pmod{p^n}$.
Similarly, for $\mathbf{z}=(z_1, z_2, \cdots z_d)^t \in \mathbb{Z}_p^d$, we let $\mathbf{z}^{(n)}=(z^{(n)}_1, z^{(n)}_2, \cdots z^{(n)}_d)^t$.

Now given $d \in \mathbb{Z}_{>0}$ and $\mathbf{z}=(z_1, z_2, \cdots z_d)^t \in \mathbb{Z}_p^d \setminus \{\mathbf{0}\}$, define a $\mathbb{Z}_p$-linear map 
$\vartheta_{\mathbf{z}}: \mathbb{Z}_p^d \to \mathbb{Z}_p$ by $$\vartheta_{\mathbf{z}}(\mathbf{w})= \mathbf{w} \cdot \mathbf{z}= \sum_{i=1}^d w_iz_i,$$
where $\mathbf{w}=(w_1,w_2, \cdots, w_d)^t \in \mathbb{Z}_p^d$.
For each $n \geqslant 0$, we can put $\vartheta_{\mathbf{z}}^{(n)}: \mathbb{Z}^{d} \to \mathbb{Z}/p^{n}\mathbb{Z}$ (note that $\mathbb{Z}^d \subset \mathbb{Z}_p^d$) such that $$\vartheta_{\mathbf{z}}^{(n)}(\mathbf{w}) \equiv \vartheta_{\mathbf{z}}(\mathbf{w}) \pmod{p^n}.$$

Note that $\vartheta_{\mathbf{z}}^{(n)}$ is $\mathbb{Z}$-linear.
Let $$G_{\mathbf{z},n}:= \ker{\vartheta_{\mathbf{z}}^{(n)}}=\{\mathbf{w} \in \mathbb{Z}^d: \vartheta_{\mathbf{z}}^{(n)}(\mathbf{w})=0 \}.$$
It is easy to see that $\mathfrak{G}_{\mathbf{z}} = (G_{\mathbf{z},n})_{n \geqslant 0}$ is a decreasing sequence of subgroups of $\mathbb{Z}^d$ with $$[\mathbb{Z}^d:G_{\mathbf{z},n}]= \# \vartheta_{\mathbf{z}}^{(n)}(\mathbb{Z}^d) \leqslant p^n< \infty,$$
and $$\lim_{n \to \infty}[\mathbb{Z}^d:G_{\mathbf{z},n}]= \lim_{n \to \infty} \# \vartheta_{\mathbf{z}}^{(n)}(\mathbb{Z}^d)= \# \vartheta_{\mathbf{z}}(\mathbb{Z}^d)=\infty,$$
since $\mathbf{z} \neq \mathbf{0}$.
Thus we obtain a $\mathbb{Z}^d$-odometer $(X_{\mathbf{z}}, \sigma_{\mathbf{z}}):=\left(X_{\mathfrak{G}_{\mathbf{z}}}, \sigma_{\mathfrak{G}_{\mathbf{z}}} \right)$.

Note that $\vartheta_{\mathbf{z}}$ is surjective if and only if $\vartheta_{\mathbf{z}}^{(n)}$ is surjective for all $n \geqslant 1$, which happens precisely when $v_p(\mathbf{z})=0$.

\begin{proposition} \label{Proposition_construction_of_algebraic_model}
    $(X_{\mathbf{z}}, \sigma_{\mathbf{z}})$ is an algebraic model for $(\mathbb{Z}_p, T_{\mathbf{z}})$ when $v_p(\mathbf{z})=0$.
\end{proposition}

\begin{proof}
    For $(\mathbf{v}_n+G_{\mathbf{z},n})_{n \geqslant 1} \in X_{\mathbf{z}}$ with $\mathbf{v}_n \in \mathbb{Z}^d$, define 
    $$\Phi\left((\mathbf{v}_n+G_{\mathbf{z},n})_{n \geqslant 1}\right)= \lim^p_{n \to \infty } \vartheta_{\mathbf{z}}^{(n)}(\mathbf{v}_n),$$
    where the limit is taken in $\mathbb{Z}_p$.
    We then check $\Phi: X_{\mathbf{z}} \to \mathbb{Z}_p$ is a (well-defined) topological conjugacy between $(X_{\mathbf{z}}, \sigma_{\mathbf{z}})$ and $(\mathbb{Z}_p, T_{\mathbf{z}})$ step by step.

    \emph{Well-definedness.}
    If $\mathbf{v}_n+G_{\mathbf{z},n} =\mathbf{w}_n+G_{\mathbf{z},n}$, then we see $\mathbf{v}_n - \mathbf{w}_n \in G_{\mathbf{z},n}$ and thus $$\vartheta_{\mathbf{z}}^{(n)}(\mathbf{v}_n-\mathbf{w}_n) \equiv 0 \pmod{p^n}.$$
    Since $\vartheta_{\mathbf{z}}^{(n)}$ is $\mathbb{Z}$-linear, we see $$\vartheta_{\mathbf{z}}^{(n)}(\mathbf{v}_n) \equiv \vartheta_{\mathbf{z}}^{(n)}(\mathbf{w}_n) \pmod{p^n}.$$
    It then follows that the limit (if existing) does not depend on the choice of representatives.

    \emph{Existence of the limit.} 
    Since $(\mathbf{v}_n+G_{\mathbf{z},n})_{n \geqslant 1} \in X_{\mathbf{z}}$, we see $\mathbf{v}_{n+1}+G_{\mathbf{z},n+1} \subset \mathbf{v}_n+G_{\mathbf{z},n}$ for each $n \geqslant 1$.
    That is $\mathbf{v}_{n+1}+G_{\mathbf{z},n} = \mathbf{v}_n +G_{\mathbf{z},n}$, so we have 
    \begin{equation} \label{Equation_compatible_relation}
        \vartheta_{\mathbf{z}}^{(n)}(\mathbf{v}_n) \equiv \vartheta_{\mathbf{z}}^{(n)}(\mathbf{v}_{n+1}) \equiv \vartheta_{\mathbf{z}}^{(n+1)}(\mathbf{v}_{n+1}) \pmod{p^n}, \quad \forall n \geqslant 1.
    \end{equation}
    Consequently, the limit exists.

    \emph{Surjectivity.} 
    Since $v_p(\mathbf{z})=0$, we know $\vartheta_{\mathbf{z}}: \mathbb{Z}_p^d \to \mathbb{Z}_p$ is surjective.
    Let $y \in \mathbb{Z}_p$. 
    There is $\mathbf{x} \in \mathbb{Z}_p^d$ such that $\vartheta_{\mathbf{z}}(\mathbf{x})=y$.
    Note that $\mathbf{x}^{(n+1)}-\mathbf{x}^{(n)} \equiv 0 \pmod{p^n}$, which gives
    $$\vartheta_{\mathbf{z}}\left(\mathbf{x}^{(n+1)}-\mathbf{x}^{(n)} \right) \equiv 0 \pmod{p^n}.$$
    It follows that $\mathbf{x}^{(n+1)}-\mathbf{x}^{(n)} \in G_{\mathbf{z},n}$ for all $n \geqslant 1$ and thus $(\mathbf{x}^{(n)}+G_{\mathbf{z},n})_{n \geqslant 1} \in X_{\mathbf{z}}$.
    Moreover, $$\vartheta_{\mathbf{z}}^{(n)}( \mathbf{x}^{(n)}) \equiv \vartheta_{\mathbf{z}}(\mathbf{x}^{(n)}) \equiv \vartheta_{\mathbf{z}}(\mathbf{x}) =y \pmod{p^n}, \quad \forall n \geqslant 1.$$
    Hence,    
    $$\Phi \left((\mathbf{x}^{(n)}+G_{\mathbf{z},n})_{n \geqslant 1}\right)= \lim^p_{n \to \infty} \vartheta_{\mathbf{z}}^{(n)}( \mathbf{x}^{(n)})= y.$$
    Namely, $\Phi$ is surjective.

    \emph{Injectivity.} 
    Let $(\mathbf{v}_n+G_{\mathbf{z},n})_{n \geqslant 1}, (\mathbf{w}_n+G_{\mathbf{z},n})_{n \geqslant 1} \in X_{\mathbf{z}}$ such that $$\Phi((\mathbf{v}_n+G_{\mathbf{z},n})_{n \geqslant 1})=\Phi((\mathbf{w}_n+G_{\mathbf{z},n})_{n \geqslant 1}).$$
    By definition, we see $$\lim^p_{n \to \infty } \vartheta_{\mathbf{z}}^{(n)}(\mathbf{v}_n)= \lim^p_{n \to \infty } \vartheta_{\mathbf{z}}^{(n)}(\mathbf{w}_n).$$
    Since $\vartheta_{\mathbf{z}}^{(n)}(\mathbf{v}_n), \vartheta_{\mathbf{z}}^{(n)}(\mathbf{w}_n) \in \mathbb{Z}/ p^n \mathbb{Z}$, by equation (\ref{Equation_compatible_relation})
    we have $$\vartheta_{\mathbf{z}}^{(n)}(\mathbf{v}_n) \equiv \vartheta_{\mathbf{z}}^{(n)}(\mathbf{w}_n) \pmod{p^n}, \quad \forall n \geqslant 1.$$
    Thus, by $\mathbb{Z}$-linearity we see $\vartheta_{\mathbf{z}}^{(n)}(\mathbf{v}_n-\mathbf{w}_n) \equiv 0 \pmod{p^n}$; i.e., $\mathbf{v}_n-\mathbf{w}_n \in G_{\mathbf{z},n}$, for all $n \geqslant 1$.
    It then follows that $(\mathbf{v}_n+G_{\mathbf{z},n})_{n \geqslant 1}=(\mathbf{w}_n+G_{\mathbf{z},n})_{n \geqslant 1}$ in $X_{\mathbf{z}}$ and $\Phi$ is injective.

    \emph{Homeomorphism.} For $N \geqslant 1$ and $\mathbf{v} \in \mathbb{Z}^d$, define the cylinder set $$\mathcal{I}_{\mathbf{v},N}:=\left\{ (\mathbf{w}_n+G_{\mathbf{z},n})_{n \geqslant 1} \in X_{\mathbf{z}}: \mathbf{w}_N+G_{\mathbf{z},N}= \mathbf{v}+G_{\mathbf{z},N}\right\}.$$
    Note that $\mathcal{I}_{\mathbf{v},N}$ is indeed a cylinder set since the first $N-1$ coordinates of the sequences in it are also restricted by projection relation.
    We claim that $\Phi(\mathcal{I}_{\mathbf{v},N})= \vartheta_{\mathbf{z}}(\mathbf{v})+p^N \mathbb{Z}_p$ for all $N \geqslant 1, \mathbf{v} \in \mathbb{Z}^d$.
    The proof of this claim is similar to the proofs of surjectivity and injectivity, hence omitted.
    The claim immediately implies $\Phi$ is a homeomorphism since it is bijective and maps cylinder sets to cylinder sets.

    \emph{Conjugacy.}
    For $\mathbf{k} \in \mathbb{Z}^d$ and $ (\mathbf{v}_n+ G_{\mathbf{z},n})_{n \geqslant 1} \in X_{\mathbf{z}}$, we see  
     \begin{align*}
        \Phi \circ \sigma_{\mathbf{z}}^{\mathbf{k}} ((\mathbf{v}_n+ G_{\mathbf{z},n})_{n \geqslant 1}) & = \Phi ((\mathbf{v}_n+\mathbf{k}+G_{\mathbf{z},n})_{n \geqslant 1})\\
        & = \lim^p_{n \to \infty} \vartheta_{\mathbf{z}}^{(n)}(\mathbf{v}_n+\mathbf{k}) \\
        & = \lim^p_{n \to \infty} \vartheta_{\mathbf{z}}^{(n)}(\mathbf{v}_n)+ \lim^p_{n \to \infty} \vartheta_{\mathbf{z}}^{(n)}(\mathbf{k}) \\
        & = \Phi( (\mathbf{v}_n+ G_{\mathbf{z},n})_{n \geqslant 1} ) + \mathbf{z \cdot k} \\
        & = T_{\mathbf{z}}^{\mathbf{k}} \circ \Phi \left( (\mathbf{v}_n+ G_{\mathbf{z},n})_{n \geqslant 1} \right)
    \end{align*}
    
    Hence, $(\mathbf{X}_{\mathbf{z}}, \sigma_{\mathbf{z}})$ is topologically conjugate to $(\mathbb{Z}_p, T_{\mathbf{z}})$ via $\Phi$ and the assertion follows.
\end{proof}

\begin{remark}
    More generally, one can show that, for $\mathbf{z}=(z_1, z_2, \cdots z_d)^t \in \mathbb{Z}_p^d \setminus \{\mathbf{0}\}$, the following three $\mathbb{Z}^d$-systems are all topologically conjugate to each other:
    $$(X_{\mathbf{z}}, \sigma_{\mathbf{z}})\cong (p^{v_p(\mathbf{z})}\mathbb{Z}_p, T_{\mathbf{z}})\cong  (\mathbb{Z}_p, T_{\hat{\mathbf{z}}}),$$
    where $\hat{\mathbf{z}}:= p^{-v_p(\mathbf{z})} \mathbf{z} \in \mathbb{Z}_p^d \setminus p \mathbb{Z}_p^d$.
\end{remark}

\begin{corollary}
    Let $\mathbf{z}=(z_1, z_2, \cdots z_d)^t \in \mathbb{Z}_p^d$ such that $v_p(\mathbf{z})=0$.
    Then the $\mathbb{Z}^d$-odometer $(X_{\mathbf{z}}, \sigma_{\mathbf{z}})$ is free if and only if $\mathbf{z}$ is $\mathbb{Z}$-independent.
\end{corollary}

\subsection{Proof of Theorem \ref{Theorem_necessary_condition_for_continuous_orbit_equivalence}}
We now use the algebraic model $(X_{\mathbf{z}}, \sigma_{\mathbf{z}})$ to compute the first cohomology group of $(\mathbb{Z}_p, T_{\mathbf{z}})$.

\begin{proposition} \label{Proposition_determining_first_cohomology_group}
    Let $\mathbf{z}=(z_1,z_2 \cdots, z_d)^t \in \mathbb{Z}_p^d$ be $\mathbb{Z}$-independent with $v_p(\mathbf{z})=0$.
    Then we have $$\mathcal{H}(\mathbb{Z}_p, T_{\mathbf{z}})= \bigcup_{n \geqslant 1}\left( \mathbb{Z}^d + \frac{1}{p^n}\mathbb{Z} \mathbf{z}^{(n)} \right)=\mathbb{Z}^d+\left\{\frac{k}{p^n}\left(z_1^{(n)}, z_2^{(n)}, \cdots , z_d^{(n)}\right)^t: k \in \mathbb{Z}, n \in \mathbb{Z}_{>0} \right\}.$$
\end{proposition}

\begin{proof}
    According to Proposition \ref{Proposition_construction_of_algebraic_model}, we have $$\mathcal{H}(\mathbb{Z}_p, T_{\mathbf{z}})=\mathcal{H}(X_{\mathbf{z}}, \sigma_{\mathbf{z}})= \bigcup_{n \geqslant 1} G_{\mathbf{z},n}^*.$$
    It then suffices to show that $G_{\mathbf{z},n}^*=\mathbb{Z}^d + p^{-n}\mathbb{Z} \mathbf{z}^{(n)}$ for all $n \geqslant 1$.
    Let $\mathbf{w} \in G_{\mathbf{z},n}=\ker{\vartheta_{\mathbf{z}}^{(n)}}$. 
    Clearly, we have 
    $$\mathbf{z}^{(n)} \cdot \mathbf{w} \equiv \mathbf{z} \cdot \mathbf{w} = \vartheta_{\mathbf{z}}(\mathbf{w}) \equiv 0 \pmod{p^n}. $$
    Hence, 
    $$\left(\mathbf{a}+\frac{k}{p^n}\mathbf{z}^{(n)} \right) \cdot \mathbf{w}\equiv \mathbf{a} \cdot \mathbf{w}+\frac{k}{p^n} \left(\mathbf{z}^{(n)} \cdot \mathbf{w}\right) \in \mathbb{Z},$$
    for all $\mathbf{a} \in \mathbb{Z}^d, k \in\mathbb{Z}, \mathbf{w} \in G_{\mathbf{z},n}$.
    By definition, $ \mathbb{Z}^d + p^{-n}\mathbb{Z} \mathbf{z}^{(n)} \subset G_{\mathbf{z},n}^*$.
    To show $G_{\mathbf{z},n}^*  \subset \mathbb{Z}^d + p^{-n}\mathbb{Z} \mathbf{z}$, it suffices to show that $[G_{\mathbf{z},n}^*: \mathbb{Z}^d]= p^n$.
    Note that $G_{\mathbf{z},n}^* / \mathbb{Z}^d$ is isomorphic to the character group of $\mathbb{Z}^d / G_{\mathbf{z},n}$.
    By surjectivity of $\vartheta_{\mathbf{z}}^{(n)}$, we know $$[G_{\mathbf{z},n}^*: \mathbb{Z}^d]=|\mathbb{Z}^d / G_{\mathbf{z},n}|= \# \vartheta_{\mathbf{z}}^{(n)}(\mathbb{Z}^d)= p^n.$$
    The proof is now finished.
\end{proof}

In view of Proposition \ref{Proposition_determining_first_cohomology_group}, it seems more concise to work in the local field $\mathbb{Q}_p$.
Formally, we denote by $\mathcal{H}_p (\mathbb{Z}_p, T_{\mathbf{z}})$ the (topological) closure of $\mathcal{H} (\mathbb{Z}_p,T_{\mathbf{z}})$ in $\mathbb{Q}_p^d$.

\begin{lemma} \label{Lemma_determining_p_adic_closure_of_first_cohomology_group}
    Let $\mathbf{z}=(z_1,z_2 \cdots, z_d)^t \in \mathbb{Z}_p^d$ such that $v_p(\mathbf{z})=0$.
    Then we have $$\mathcal{H}_p(\mathbb{Z}_p, T_{\mathbf{z}})= \mathbb{Z}_p^d+ \mathbb{Q}_p \mathbf{z}=\mathbb{Z}_p^d+\{\lambda\mathbf{z}: \lambda \in \mathbb{Q}_p\}.$$
    Consequently, $\mathcal{H}_p (\mathbb{Z}_p, T_{\mathbf{z}})$ contains a unique non-trivial $\mathbb{Q}_p$-linear subspace $\mathbb{Q}_p \mathbf{z}$.
\end{lemma}

\begin{proof}
    The computation of $\mathcal{H}_p(\mathbb{Z}_p, T_{\mathbf{z}})$ follows immediately from Proposition \ref{Proposition_determining_first_cohomology_group} and thus $\mathbb{Q}_p \mathbf{z}$ is obviously a non-trivial (since $\mathbf{z} \neq \mathbf{0}$) $\mathbb{Q}_p$-linear subspace.
    Now we prove the uniqueness part by contradiction.
    Suppose there is $\mathbf{w} \notin \mathbb{Q}_p \mathbf{z}$ such that $\mathbb{Q}_p \mathbf{w} \subset \mathcal{H}_p (\mathbb{Z}_p, T_{\mathbf{z}})=\mathbb{Z}_p^d+\mathbb{Q}_p \mathbf{z}$.
    Put $$D:=\mathrm{dist}(\mathbf{w}, \mathbb{Q}_p \mathbf{z})=\inf_{ \lambda \in \mathbb{Q}_p } \mathrm{dist}(\mathbf{w}, \lambda \mathbf{z}),$$
    where $\mathrm{dist}$ is the distance on $\mathbb{Q}_p^d$ (see subsection 2.1).
    A standard argument gives that there is $\hat{\lambda} \in \mathbb{Q}_p$ such that $D= \mathrm{dist}(\mathbf{w}, \hat{\lambda} \mathbf{z})>0$.
    Take $\bar{\lambda} \in \mathbb{Q}_p$ with $|\bar{\lambda}|_p> d/D$.
    Then we see $$\mathrm{dist}(\bar{\lambda} \mathbf{w}, \mathbb{Q}_p \mathbf{z})=\mathrm{dist}(\bar{\lambda} \mathbf{w}, \bar{\lambda} \hat{\lambda} \mathbf{z})=|\bar{\lambda}|_p \cdot D>d.$$
    This contradicts that $\bar{\lambda} \mathbf{w} \in \mathbb{Q}_p \mathbf{w} \subset \mathbb{Z}_p^d+\mathbb{Q}_p \mathbf{z}$ since every vector in $\mathbb{Z}_p^d+\mathbb{Q}_p \mathbf{z}$ is $d$-close to $\mathbb{Q}_p \mathbf{z}$.
\end{proof}

\begin{proof}[Proof of Theorem \ref{Theorem_necessary_condition_for_continuous_orbit_equivalence}]
Suppose that $(\mathbb{Z}_p, T_{\mathbf{z}})$ and $(\mathbb{Z}_p, T_{\mathbf{w}})$ are continuously orbit equivalent.
By Lemma \ref{Lemma_continuous_orbit_equivalence_for_Z^d_odometers}, there is $\mathbf{P} \in \mathrm{GL}_{d}(\mathbb{Q})$ with $\det(\mathbf{P})=\pm1$ such that $$\mathcal{H}(\mathbb{Z}_p, T_{\mathbf{w}})=\mathbf{P} \mathcal{H} (\mathbb{Z}_p, T_{\mathbf{z}}).$$
Taking $p$-adic closure yields $$\mathcal{H}_p(\mathbb{Z}_p, T_{\mathbf{w}})=\mathbf{P} \mathcal{H}_p(\mathbb{Z}_p, T_{\mathbf{z}}).$$
Since $\mathbf{P}$ is invertible, by Lemma \ref{Lemma_determining_p_adic_closure_of_first_cohomology_group} (note that now we have $v_p(\mathbf{z})=v_p(\mathbf{w})=0$), we have
$$\mathbb{Q}_p \mathbf{w}= \mathbf{P} (\mathbb{Q}_p \mathbf{z})=\mathbb{Q}_p \mathbf{Pz}$$
Hence we see $\mathbf{w}=\lambda \mathbf{P} \mathbf{z}$ for some $\lambda \in \mathbb{Q}_p^{\times}$.
\end{proof}

\section{Free minimal constant speedups for $(\mathbb{Z}_p,T_{\mathbf{z}})$}

\subsection{Constant speedups for algebraic models}

This subsection is independent of the next subsection. 
So there is no loss of completeness for the proof of the main result if the reader skips this subsection.

\begin{proposition} \label{Proposition_minimal_and_free_condition_for_speed_matrix_C}
    Let $\mathfrak{G}=(G_n)_{n\geqslant1}$ be a decreasing sequence of finite-index subgroups of $\mathbb{Z}^d$ such that $\lim_{n \to \infty}[\mathbb{Z}^d:G_n] =\infty$ and $\cap_{n \geqslant 1}G_n = \{ \mathbf{0} \}$.
    Let $S: \mathbb{Z}^d \curvearrowright X_{\mathfrak{G}}$ be a constant speedup of $(X_{\mathfrak{G}}, \sigma_{\mathfrak{G}})$ with speed $\mathbf{C} \in \mathrm{Mat}_{d \times d}(\mathbb{Z})$.
    Then we have:

    (1) $(X_{\mathfrak{G}},S)$ is minimal if and only if $\mathbf{C}\mathbb{Z}^d + G_n = \mathbb{Z}^d$ for all $n \geqslant 1$.

    (2) $(X_{\mathfrak{G}},S)$ is free if and only if $\det(\mathbf{C}) \neq 0$.

    (3) Suppose $(X_{\mathfrak{G}},S)$ is minimal and free.
    Put $H_n=\mathbb{Z}^d \cap \mathbf{C}^{-1}G_n$ for each $n \geqslant 1$.
    Then $\mathfrak{H}:=(H_n)_{n \geqslant 1}$ is a decreasing sequence of finite-index subgroups of $\mathbb{Z}^d$ such that $\lim_{n \to \infty}[\mathbb{Z}^d:H_n] =\infty$ and $\cap_{n \geqslant 1}H_n = \{ \mathbf{0} \}$.
    Moreover, $(X_{\mathfrak{G}},S)$ is topologically conjugate to $(X_{\mathfrak{H}}, \sigma_{\mathfrak{H}})$.
\end{proposition}

\begin{proof}
    Note that $(X_{\mathfrak{G}},S)$ is minimal if and only if the projective action $S_n: \mathbb{Z}^d \curvearrowright \mathbb{Z}^d/G_n $ is minimal for all $n \geqslant 1$, where $$S_n^{\mathbf{v}}(\mathbf{w}+G_n):= S^{\mathbf{v}}(\mathbf{w})+ G_n = \sigma_{\mathfrak{G}}^{\mathbf{Cv}}(\mathbf{w})+G_n= \mathbf{Cv}+\mathbf{w}+G_n.$$
    So (1) follows by noting that $S_n$ is minimal if and only if $\mathbf{C}\mathbb{Z}^d+G_n= \mathbb{Z}^d$.

    Now we prove (2). We see $\det(\mathbf{C}) = 0$ if and only if there is $\mathbf{v} \in \mathbb{Z}^d \setminus \{\mathbf{0}\}$ such that $\mathbf{Cv=0}$.
    So the assertion follows immediately by Proposition \ref{Proposition_free_and_nonvanishing}.

    To see (3), one can follow the proof of Theorem 5.3 in \cite{Topological_speedups_of_Z^d_actions}, with correction to its Claim 2.
\end{proof}

\begin{corollary} \label{Corollary_computation_of_speedup_of_p_adic_Z^d_odometer}
    Let $\mathbf{z} \in \mathbb{Z}_p^d$ be $\mathbb{Z}$-independent with $v_p(\mathbf{z})=0$.
    Let $S: \mathbb{Z}^d \curvearrowright X_{\mathbf{z}}$ be a constant speedup of $(X_{\mathbf{z}}, \sigma_{\mathbf{z}})$ with speed $\mathbf{C} \in \mathrm{Mat}_{d \times d}(\mathbb{Z})$.
    Then we have

    (1) $(X_{\mathbf{z}},S)$ is minimal if and only if $v_p\left(\mathbf{C}^t \mathbf{z}\right)=0.$

    (2) $(X_{\mathbf{z}},S)$ is free if and only if $\det(\mathbf{C}) \neq 0$.

    (3) When $(X_{\mathbf{z}},S)$ is minimal and free, it is topologically conjugate to $(X_{\mathbf{w}}, \sigma_{\mathbf{w}})$ with $\mathbf{w}:=\mathbf{C}^t \mathbf{z}$.
\end{corollary}
\begin{proof}
    (1) By Proposition \ref{Proposition_minimal_and_free_condition_for_speed_matrix_C} (1), $(X_{\mathbf{z}},S)$ is minimal if and only if $$\mathbf{C} \mathbb{Z}^d + G_{\mathbf{z},n} = \mathbb{Z}^d, \quad \forall n \geqslant 1.$$
    This is equivalent to $$ \vartheta_{\mathbf{z}}^{(n)} ( \mathbf{C} \mathbb{Z}^d) =\vartheta_{\mathbf{z}}^{(n)} (\mathbb{Z}^d)= \mathbb{Z}/ p^n \mathbb{Z}, \quad \forall n \geqslant 1.$$ 
    That is, $\vartheta_{\mathbf{z}} ( \mathbf{C} \mathbb{Z}_p^d) = \mathbb{Z}_p$.
    Note that $$\vartheta_{\mathbf{z}} (\mathbf{Cv})=(\mathbf{Cv}) \cdot \mathbf{z}= \mathbf{v} \cdot (\mathbf{C}^t \mathbf{z})=\vartheta_{\mathbf{C}^t \mathbf{z}}(\mathbf{v}), \quad \forall \mathbf{v} \in \mathbb{Z}_p^d.$$
    Thus, we see $\vartheta_{\mathbf{z}} ( \mathbf{C} \mathbb{Z}_p^d)= \vartheta_{\mathbf{C}^t \mathbf{z}}(\mathbb{Z}_p^d)$.
    The first assertion follows from the fact that $\vartheta_{\mathbf{C}^t \mathbf{z}}(\mathbb{Z}_p^d)= \mathbb{Z}_p$ if and only if $\mathbf{C}^t \mathbf{z} \notin p \mathbb{Z}_p^d$.
    
    (2) See Proposition \ref{Proposition_minimal_and_free_condition_for_speed_matrix_C} (2).

    (3) According to Proposition \ref{Proposition_minimal_and_free_condition_for_speed_matrix_C} (3), it suffices to check that $G_{\mathbf{w},n}= \mathbb{Z}^d \cap \mathbf{C}^{-1} G_{\mathbf{z},n} $ for all $n \geqslant 1$.
    Firstly, take $\mathbf{C}^{-1} \mathbf{v} \in \mathbb{Z}^d \cap \mathbf{C}^{-1} G_{\mathbf{z},n}$ with $\mathbf{v} \in G_{\mathbf{z},n}$.
    Then we see $$\vartheta_{\mathbf{w}}(\mathbf{C}^{-1} \mathbf{v})= \mathbf{w} \cdot (\mathbf{C}^{-1} \mathbf{v}) = (\mathbf{C}^t \mathbf{z}) \cdot (\mathbf{C}^{-1} \mathbf{v}) = \mathbf{z} \cdot \mathbf{v} \equiv 0  \pmod{p^n}. $$
    It means $\mathbf{C}^{-1} \mathbf{v} \in G_{\mathbf{w},n}$ and hence $\mathbb{Z}^d \cap \mathbf{C}^{-1} G_{\mathbf{z},n} \subset G_{\mathbf{w},n}$.

    Conversely, for $\mathbf{u} \in G_{\mathbf{w},n}$, we put $\mathbf{u}':=\mathbf{Cu}$.
    Note that $$\vartheta_{\mathbf{z}}(\mathbf{u}')= \mathbf{z} \cdot \mathbf{u}'= \mathbf{z} \cdot \mathbf{Cu}= \mathbf{w} \cdot \mathbf{u} \equiv 0 \pmod{p^n}.$$
    Thus $\mathbf{u}' \in G_{\mathbf{z},n}$ and hence $\mathbf{u}=\mathbf{C}^{-1} \mathbf{u}' \in \mathbb{Z}^d \cap \mathbf{C}^{-1} G_{\mathbf{z},n}$.
\end{proof}

\subsection{Proof of the main result}

We will prove Theorem A in this subsection.
To begin with, we have the following easy fact.

\begin{proposition} \label{Proposition_speedu_of_p_adic_Z^d_odometer_of_adding_type}
    Let $\mathbf{z} \in \mathbb{Z}_p^d$ be $\mathbb{Z}$-independent with $v_p(\mathbf{z})=0$.
    Let $S: \mathbb{Z}^d \curvearrowright \mathbb{Z}_p$ be a constant speedup of $(\mathbb{Z}_p, T_{\mathbf{z}})$ with speed $\mathbf{C} \in \mathrm{Mat}_{d \times d}(\mathbb{Z})$.
    Then $S=T_{\mathbf{w}}$ with $\mathbf{w}:=\mathbf{C}^t\mathbf{z}$.
    Consequently, $(\mathbb{Z}_p, S)$ is minimal if and only if $v_p(\mathbf{C}^t\mathbf{z})=0$;
     $(\mathbb{Z}_p, S)$ is free if and only if $\mathbf{C}^t\mathbf{z}$ is $\mathbb{Z}$-independent, which happens precisely when $\det(\mathbf{C}) \neq 0$.
\end{proposition}

\begin{proof}
    For $\mathbf{n} \in \mathbb{Z}^d$ and $x \in \mathbb{Z}_p$, we see $$S^{\mathbf{n}}(x)= T_{\mathbf{z}}^{\mathbf{Cn}}(x)=x+\mathbf{(Cn) \cdot z}=x+\mathbf{n} \cdot \mathbf{C}^t \mathbf{z}= T_{\mathbf{w}}^{\mathbf{n}}(x).$$
    So we have $S=T_{\mathbf{w}}$ and the remaining assertions are then all clear.
\end{proof}

For a prime $p$ and a positive integer $d$, we put 
\begin{equation*}
\mathcal{E}_{p,d}:=
\left\{ (\mathbf{z}, \mathbf{C}) \in \mathbb{Z}_p^d \times \mathrm{Mat}_{d \times d}(\mathbb{Z}) \;\middle|\;
\begin{aligned}
    &v_p(\mathbf{z})=v_p(\mathbf{C}^t \mathbf{z})=0;\\
    &\det(\mathbf{C}) \neq 0\ \mathrm{and}\ \mathbf{z}\ \mathrm{ is }\ \mathbb{Z}\mathrm{-independent}; \\
    & \mathbf{C}^t \mathbf{z} \neq  \lambda \mathbf{Pz}, \ \forall \lambda \in \mathbb{Q}_p^{\times}, \mathbf{P} \in \mathrm{GL}_d(\mathbb{Q})\ \mathrm{with} \ \det(\mathbf{P})=\pm1.
\end{aligned}
 \right\}.
\end{equation*}
We sometimes simply write $\mathcal{E}$ or $\mathcal{E}_d$ for $\mathcal{E}_{p,d}$.

In view of Theorem \ref{Theorem_necessary_condition_for_continuous_orbit_equivalence} and Proposition \ref{Proposition_speedu_of_p_adic_Z^d_odometer_of_adding_type}, we know Conjecture \ref{Conjecture_continuousorbit_equivalence_for_free_speedups} is false if $\mathcal{E}$ is non-empty.
Our main result (i.e., Theorem A or Theorem \ref{Theorem_main_theorem_residuality_of_exception_set} below) says that $\mathcal{E}$ indeed has at least one (somehow) large section for each $d \geqslant 2$.
To be more precise, for $\mathbf{C} \in \mathrm{Mat}_{d \times d}(\mathbb{Z})$ with $\det(\mathbf{C}) \neq 0$, we put $\mathcal{E}^{\mathbf{C}}_{p,d}:=\{\mathbf{z} \in \mathbb{Z}_p^d: (\mathbf{z,C}) \in \mathcal{E}_{p,d} \}$.
The rest of this subsection aims to find a good $\mathbf{C}$ such that the section $\mathcal{E}^{\mathbf{C}}_{p,d}$ is (somehow) large.
To do this, we have to deal with the three conditions proposed in the definition of $\mathcal{E}_{p,d}$.

We first handle the last condition.
For $\mathbf{C} \in \mathrm{Mat}_{d \times d}(\mathbb{Z})$ with $\det(\mathbf{C}) \neq 0$, we put
$$\widehat{\mathcal{E}}_{p,d}^{\mathbf{C}}:=\left\{ \mathbf{z} \in \mathbb{Q}_p^d : \mathbf{C}^t \mathbf{z} \neq  \lambda \mathbf{Pz}, \ \forall \lambda \in \mathbb{Q}_p^{\times}, \mathbf{P} \in \mathrm{GL}_d(\mathbb{Q})\ \mathrm{with} \ \det(\mathbf{P})= \pm1 \right\}.$$

\begin{lemma} \label{Lemma_residuality_of_absence_of_necessary_condidtion}
    Let $\mathbf{C} \in \mathrm{Mat}_{d \times d}(\mathbb{Z})$ with $\det(\mathbf{C}) \neq 0$.
    Then $\widehat{\mathcal{E}}_{p,d}^{\mathbf{C}} \neq \varnothing$ if and only if $\mathbf{C}^t \neq \lambda \mathbf{P}$ for all $\lambda \in \mathbb{Q}^{\times}$ and $\mathbf{P} \in \mathrm{GL}_d(\mathbb{Q})$ with $\det(\mathbf{P})= \pm1$.
    Moreover, when $\widehat{\mathcal{E}}_{p,d}^{\mathbf{C}} \neq \varnothing$, it is residual in $\mathbb{Q}_p^d$; i.e., it contains a countable intersection of open dense subsets of $\mathbb{Q}_p^d$.
\end{lemma}

\begin{proof}
    For each fixed $\mathbf{P} \in \mathrm{GL}_d(\mathbb{Q})$ with $\det(\mathbf{P})= \pm1$, we see that 
    $$\left\{ \mathbf{z} \in \mathbb{Q}_p^d  :\ \exists \lambda \in \mathbb{Q}_p^{\times} \ \mathrm{s.t.} \ \mathbf{C}^t \mathbf{z} =  \lambda \mathbf{Pz} \right\}= \{\mathbf{0}\} \cup \bigcup_{\lambda \in \mathrm{Spec}_{\mathbb{Q}_p}(\mathbf{P}^{-1} \mathbf{C}^t)}\mathrm{Eig}_{\mathbb{Q}_p}(\lambda, \mathbf{P}^{-1} \mathbf{C}^t),$$
    where $\mathrm{Spec}_{\mathbb{Q}_p}(\mathbf{P}^{-1} \mathbf{C}^t)$ is the set of all eigenvalues of $\mathbf{P}^{-1} \mathbf{C}^t$ in $\mathbb{Q}_p$ and $\mathrm{Eig}_{\mathbb{Q}_p}(\lambda, \mathbf{P}^{-1} \mathbf{C}^t)$ is the eigenspace of $\mathbf{P}^{-1} \mathbf{C}^t$ in $\mathbb{Q}_p^d$ w.r.t. the eigenvalue $\lambda$.
    So we have $$\widehat{\mathcal{E}}_{p,d}^{\mathbf{C}}=  \mathbb{Q}_p^d \setminus \left(\{\mathbf{0}\} \cup \bigcup_{\substack{\mathbf{P} \in \mathrm{GL}_d(\mathbb{Q}),\\ \det(\mathbf{P})=\pm 1}} \bigcup_{\lambda \in \mathrm{Spec}_{\mathbb{Q}_p}(\mathbf{P}^{-1} \mathbf{C}^t)}\mathrm{Eig}_{\mathbb{Q}_p}(\lambda, \mathbf{P}^{-1} \mathbf{C}^t) \right).$$

    Suppose $\mathbf{C}^t = \lambda_0 \mathbf{P}_0$ for some $\lambda_0 \in \mathbb{Q}^{\times}$ and $\mathbf{P}_0 \in \mathrm{GL}_d(\mathbb{Q})$ with $\det(\mathbf{P}_0)= \pm1$.
    We see that $\mathbf{P}_0^{-1}\mathbf{C}^t = \lambda_0 I_{d}$ where $I_d$ is the $d \times d$ identity matrix.
    Hence, $$\mathrm{Eig}_{\mathbb{Q}_p}(\lambda_0, \mathbf{P}_0^{-1} \mathbf{C}^t)= \mathrm{Eig}_{\mathbb{Q}_p}(\lambda_0, \lambda_0 I_d)= \mathbb{Q}_p^d.$$
    It follows that $\widehat{\mathcal{E}}_{p,d}^{\mathbf{C}} = \varnothing$.

    Now suppose $\mathbf{C}^t \neq \lambda \mathbf{P}$ for all $\lambda \in \mathbb{Q}^{\times}$ and $\mathbf{P} \in \mathrm{GL}_d(\mathbb{Q})$ with $\det(\mathbf{P})= \pm1$.
    Since both $\mathbf{C}$ and $\mathbf{P}$ are over $\mathbb{Q}$, we see every $\mathbf{P}^{-1} \mathbf{C}^t$ is not a scalar matrix and hence $\mathrm{Eig}_{\mathbb{Q}_p}(\lambda, \mathbf{P}^{-1} \mathbf{C}^t)$ is a proper $\mathbb{Q}_p$-linear subspace of $\mathbb{Q}_p^d$ which is closed and has empty interior.
    Moreover, $\# \mathrm{Spec}_{\mathbb{Q}_p}(\mathbf{P}^{-1} \mathbf{C}^t) \leqslant d$ and there are only countably many rational $d \times d$ matrices over $\mathbb{Q}$.
    Hence, the complement of $\widehat{\mathcal{E}}_{p,d}^{\mathbf{C}}$ is a countable union of closed subsets with empty interiors.
    So $\widehat{\mathcal{E}}_{p,d}^{\mathbf{C}}$ is residual and is in particular non-empty.
\end{proof}

\begin{corollary}
    When $d=1$, we have $\widehat{\mathcal{E}}_{p,1}^{\mathbf{C}} = \varnothing$ for all $\mathbf{C} \in \mathrm{Mat}_{1 \times 1}(\mathbb{Z})$ with $\det(\mathbf{C}) \neq 0$.
\end{corollary}

We now deal with the $\mathbb{Z}$-independence condition.
Recall that $\mathbb{Z}$-independence is nothing but $\mathbb{Q}$-independence.
For a prime $p$ and a positive integer $d$, put $$\mathcal{D}_{p,d}:=\{\mathbf{z}=(z_1,z_2,...,z_d)^t \in \mathbb{Q}_p^d: z_1,z_2,...,z_d \ \mathrm{are} \ \mathbb{Q} \mathrm{-independent}\}.$$

\begin{lemma} \label{Lemma_residuality_for_independence_condition}
    The set $\mathcal{D}_{p,d}$ is residual in $\mathbb{Q}_p^d$.
\end{lemma}

\begin{proof}
    Note that $$\mathcal{D}_{p,d}=\mathbb{Q}_p^d \setminus \bigcup_{\mathbf{v} \in \mathbb{Q}^d \setminus \{ \mathbf{0}\}}\{ \mathbf{z} \in  \mathbb{Q}_p^d : \mathbf{v \cdot z}=0\},$$
    which is residual in $\mathbb{Q}_p^d$ since each set $\{ \mathbf{z} \in  \mathbb{Q}_p^d : \mathbf{v \cdot z}=0\}$ is a hyperplane in $\mathbb{Q}_p^d$ (and thus is closed and has empty interior).
\end{proof}

We simply write $$\mathcal{Z}_{p,d}:=\mathbb{Z}_p^d \setminus p \mathbb{Z}_p^d = \left\{\mathbf{z} \in \mathbb{Z}_p^d: v_p(\mathbf{z})=0 \right\},$$
which is clearly a non-empty clopen subset of $\mathbb{Z}_p^d$.
Note that $$\mathbb{Q}_p^d =\{ \mathbf{0} \} \cup \bigcup_{k \in \mathbb{Z}} p^{k} \mathcal{Z}_{p,d}. $$
Moreover, for each $\mathbf{C} \in \mathrm{Mat}_{d \times d}(\mathbb{Z})$ with $\det(\mathbf{C}) \neq 0$, we have 
\begin{equation} \label{Equation_computation_of_C_section_of_exception_set}
    \mathcal{E}^{\mathbf{C}}_{p,d}=\widehat{\mathcal{E}}_{p,d}^{\mathbf{C}} \cap \mathcal{D}_{p,d} \cap \mathcal{Z}_{p,d} \cap \mathbf{C}^{-t} \mathcal{Z}_{p,d}.
\end{equation}

Now we are ready to prove Theorem A, our main result, as follows.

\begin{theorem} [Equivalent version of Theorem A] \label{Theorem_main_theorem_residuality_of_exception_set}
    Let $d \geqslant 2$ be an integer and $p$ be a prime. 
    Then there is some $\mathbf{C} \in \mathrm{Mat}_{d \times d}(\mathbb{Z})$ with $\det(\mathbf{C}) \neq 0$ such that $\mathcal{E}^{\mathbf{C}}_{p,d}$ is residual in $\mathcal{Z}_{p,d}$.
    Consequently, $\mathcal{E}_d \neq \varnothing$ for all $d \geqslant 2$. 
\end{theorem}

\begin{proof}
    Let $r=3$ if $p=2$ and $r=2$ otherwise.
    We will proceed with the $d \times d$ diagonal matrix $$\mathbf{C}:=\mathrm{diag}(1,1, \cdots,1,r).$$

    Firstly, we see $\mathbf{C}^{-t} \mathcal{Z}_{p,d}= \mathcal{Z}_{p,d}$ since $r \in \mathbb{Z}_p^{\times}$.
    Secondly, by Lemma \ref{Lemma_residuality_for_independence_condition}, $\mathcal{D}_{p,d}$ is residual in $\mathbb{Q}_p^d$ and by definition it is scaling-invariant, meaning that $\lambda \mathcal{D}_{p,d}= \mathcal{D}_{p,d},$ for all $\lambda \in \mathbb{Q}_p^{\times}$.
    So $\mathcal{D}_{p,d} \cap \mathcal{Z}_{p,d}$ is residual in $\mathcal{Z}_{p,d}$.
    Finally, note that $x^d \pm r=0$ has no solution in $\mathbb{Q}$.
    By taking determinants, we see $\mathbf{C}^t \neq \lambda \mathbf{P}$ for all $\lambda \in \mathbb{Q}^{\times}$ and $\mathbf{P} \in \mathrm{GL}_d(\mathbb{Q})$ with $\det(\mathbf{P})=\pm1$.
    According to Lemma \ref{Lemma_residuality_of_absence_of_necessary_condidtion}, $\widehat{\mathcal{E}}_{p,d}^{\mathbf{C}}$ is residual in $\mathbb{Q}_p^d$.
    It is also scaling-invariant by definition.
    So $\widehat{\mathcal{E}}_{p,d}^{\mathbf{C}} \cap \mathcal{Z}_{p,d}$ is residual in $\mathcal{Z}_{p,d}$.

    Now in view of equation (\ref{Equation_computation_of_C_section_of_exception_set}), we see $\mathcal{E}^{\mathbf{C}}_{p,d}$ is the intersection of two residual sets in $\mathcal{Z}_{p,d}$ and thus is residual in $\mathcal{Z}_{p,d}$ as well.
    Hence, $\mathcal{E}^{\mathbf{C}}_{p,d} \neq \varnothing$, so is $\mathcal{E}_d$.
\end{proof}

\begin{remark}
    Indeed, one can show that, $\mathcal{E}^{\mathbf{C}}_{p,d} \neq \varnothing$ if and only if the following two conditions are simultaneously satisfied:

    (1) $\mathbf{C}^t \neq \lambda \mathbf{P}$ for all $\lambda \in \mathbb{Q}^{\times}$ and $\mathbf{P} \in \mathrm{GL}_d(\mathbb{Q})$ with $\det(\mathbf{P})= \pm1$;

    (2) $\mathcal{Z}_{p,d} \cap \mathbf{C}^{-t} \mathcal{Z}_{p,d} \neq \varnothing$.

    \noindent Moreover, when these two conditions hold, $\mathcal{E}^{\mathbf{C}}_{p,d}$ is residual in $\mathcal{Z}_{p,d} \cap \mathbf{C}^{-t} \mathcal{Z}_{p,d}$.
\end{remark}

\section{Discussion}

In general, our condition for continuous orbit equivalence for adding-type $p$-adic $\mathbb{Z}^d$-odometers in Theorem \ref{Theorem_necessary_condition_for_continuous_orbit_equivalence} is not sufficient.
Even for free minimal constant speedups, it is not the only obstruction to continuously orbit equivalence, as the following example indicates.
Here we give a construction for the simple case $d=2$ and we believe that the idea behind our example works more generally.
\begin{example}
    For each prime $p$, we let $\alpha=\alpha_p \in \mathbb{Z}_p$ such that $\alpha$ is transcendental over $\mathbb{Q}$ (such an element always exists since $\mathbb{Z}_p$ is uncountable). 
    Put $\mathbf{z}:=(1, \alpha)^t \in \mathbb{Z}_p^2$ which is clearly $\mathbb{Z}$-independent with $v_p(\mathbf{z})=0$ and $$\mathbf{C}=\mathbf{C}_{q}:= \begin{pmatrix}
        1 & 0 \\
         0 & q^2\\
    \end{pmatrix},$$
    where $q$ is another prime distinct from $p$.
    Then we see $\det(\mathbf{C}) = q^2 \neq 0$ and $v_p(\mathbf{w})=0$ where
     $\mathbf{w}:=\mathbf{C}^t \mathbf{z}=(1, q^2 \alpha)^t.$
    It is then clear that $\mathbf{w}$ is $\mathbb{Z}$-independent and $$\mathbf{w}= q \begin{pmatrix}
        q^{-1} & 0 \\
         0 & q\\
    \end{pmatrix} \mathbf{z}.$$

    So we see $T_{\mathbf{w}}$ is a free minimal constant speedup of $T_{\mathbf{z}}$ (with speed $\mathbf{C}$).
    Moreover, they satisfy the conclusion in Theorem \ref{Theorem_necessary_condition_for_continuous_orbit_equivalence}.
    Now we show that $T_{\mathbf{z}}$ is not continuously orbit equivalent to $T_{\mathbf{w}}$.

    We assume, to the contrary, $T_{\mathbf{z}}$ is continuously orbit equivalent to $T_{\mathbf{w}}$.
    By Lemma \ref{Lemma_continuous_orbit_equivalence_for_Z^d_odometers}, there is $$\mathbf{P}= \begin{pmatrix}
    a_1 & a_2\\
    a_3 & a_4\\
    \end{pmatrix} \in \mathrm{GL}_2(\mathbb{Q})$$ with $\det(\mathbf{P})= \pm 1$ such that
    \begin{equation} \label{Equation_example_1}
        \mathbf{P}\mathcal{H}(\mathbb{Z}_p, T_{\mathbf{z}})=\mathcal{H}(\mathbb{Z}_p, T_{\mathbf{w}}).
    \end{equation}
    By Proposition \ref{Proposition_determining_first_cohomology_group}, we see 
    \begin{equation}\label{Equation_example_2}
        \mathcal{H}(\mathbb{Z}_p, T_{\mathbf{w}}) \subset \mathbb{Z}\left[ \frac{1}{p}\right] \times \mathbb{Z}\left[\frac{1}{p} \right].
    \end{equation}
    Since $\mathbb{Z}^2 \subset \mathcal{H}(\mathbb{Z}_p, T_{\mathbf{z}})$, relations (\ref{Equation_example_1}) and (\ref{Equation_example_2}) imply
    $$ a_1,a_2,a_3,a_4 \in \mathbb{Z}\left[\frac{1}{p} \right].$$
    By the same argument as in the proof of Theorem \ref{Theorem_necessary_condition_for_continuous_orbit_equivalence} (taking $p$-adic closure in relation (\ref{Equation_example_1}) and then considering the unique non-trivial $\mathbb{Q}_p$-subspace), we obtain $$\mathbf{w}= \lambda \mathbf{P} \mathbf{z},$$
    for some $\lambda \in \mathbb{Q}_p^{\times}$.
    That is, by direct computation, we see 
    $$\begin{cases}
        \lambda(a_1+a_2 \alpha)=1, \\
        \lambda(a_3+a_4 \alpha)= q^2 \alpha.\\
    \end{cases}$$
    Eliminating $\lambda$ yields $$a_2 q^2 \alpha^2 + (a_1 q^2 -a_4)\alpha -a_3=0.$$
    Since $\alpha$ is transcendental over $\mathbb{Q}$, we see $$\begin{cases}
        a_2=a_3=0,\\
        a_1 q^2=a_4.
    \end{cases}$$
    Moreover, we see $$\pm 1=\det(\mathbf{P})=a_1a_4-a_2a_3=a_1^2 q^2,$$
    which gives $a_1=\pm q^{-1}$, contradicting the fact that $a_1 \in \mathbb{Z}\left[ \frac{1}{p} \right]$.
    
    Hence, $T_{\mathbf{z}}$ is not continuously orbit equivalent to $T_{\mathbf{w}}$.
\end{example}

This example says that there must be some other (very likely algebraic) obstruction to continuous orbit equivalence during the free minimal constant speedups of adding-type $p$-adic $\mathbb{Z}^d$-odometers.
It is then natural to ask for the mechanism for violating continuous orbit equivalence.

\begin{question}
Let $\mathbf{z}, \mathbf{w} \in \mathbb{Z}_p^d$ be $\mathbb{Z}$-independent with $v_p(\mathbf{z})=v_p(\mathbf{w})=0$.
Are there any (elegant) equivalent conditions in terms of $\mathbf{z}, \mathbf{w}$ for continuous orbit equivalence between the free $\mathbb{Z}^d$-odometers $(\mathbb{Z}_p, T_{\mathbf{z}})$ and $(\mathbb{Z}_p, T_{\mathbf{w}})$?
How about if additionally $T_{\mathbf{w}}$ is a free minimal constant speedup for $T_{\mathbf{z}}$?
\end{question}

\bigskip

\noindent\textbf{Acknowledgments.} 
The author would like to warmly thank Jia-Yan Yao for interesting discussions on the subject. He would like also to heartily thank
the National Natural Science Foundation of China (Grant No.\,12231013) for partial financial support.

\bibliographystyle{amsalpha}
\bibliography{references}

\vskip 1 cm
\begin{tabular}{ll}
Chang-Hua JIAO & \\
Department of Mathematics &  \\
Tsinghua University &  \\
Beijing 100084 &  \\
People's Republic of China &  \\
E-mail: jch23@mails.tsinghua.edu.cn &
\end{tabular}

\end{document}